\documentclass[10pt]{amsart}
\usepackage{amssymb}
\usepackage{amsthm}
\usepackage{setspace}
\usepackage{lineno}
\usepackage[pdftex, dvipsnames]{xcolor}
\usepackage[pdftex, final]{graphicx}
\usepackage{caption}
\usepackage{pinlabel}
\usepackage{nicematrix}
\usepackage{tikz}
\usepackage{hyperref}
\hypersetup{
    colorlinks=true,
    linkcolor=MidnightBlue,
    citecolor=ForestGreen,
    urlcolor=BrickRed
}

\makeatletter
\g@addto@macro\@floatboxreset\centering
\makeatother

\allowdisplaybreaks[4]

\AtBeginDocument{%
\def\MR#1{}
}

\newtheorem{theorem}{Theorem}
\newtheorem{lemma}[theorem]{Lemma}
\newtheorem{corollary}[theorem]{Corollary}

\theoremstyle{definition}

\newtheorem{definition}{Definition}
\newtheorem{example}[]{Example}

\usepackage[pdftex,%
letterpaper,%
includehead,%
includefoot,%
nomarginpar,%
lmargin=1in,%
rmargin=1in,%
tmargin=1in,%
bmargin=1in,%
]{geometry}

\newcommand{\PID}{\textup{PID}}
\newcommand{\LCM}{\operatorname{LCM}}
\newcommand*{\Z}{\mathbb{Z}}
\newcommand*{\Q}{\mathbb{Q}}
\newcommand*{\hk}{H_1(\widetilde{X}_K, \Q)}
\newcommand*{\qt}{\Q[t, t^{-1}]}
\DeclareMathOperator{\GL}{\operatorname{GL}}

\begin{document}
\title{The computation of higher order Alexander invariants}

\author{Charles Livingston}
\address{Department of Mathematics, Indiana University, Bloomington, IN 47405}
\email{\href{mailto:livingst@iu.edu}{livingst@iu.edu}}

\begin{abstract} We present an algorithm that efficiently computes higher order Alexander polynomials, $\Delta_{K,i}(t) \in \Z[t]$ ($i \ge 1$), for knots with as many as 100 crossings. These polynomials were first defined by Alexander in 1928. The first, $\Delta_{K,1}(t)$, is \textit{the} Alexander polynomial. From a modern perspective, these are normalized products of invariant factors of the rational homology of the infinite cyclic cover of the complement of $K$ regarded as a module over $\Q[t,t^{-1}]$.

In theory, computing higher order Alexander polynomials can be done with elementary linear algebra methods that convert a presentation matrix for the homology into Smith normal form. In practice, this becomes computationally infeasible for   knots with as few as 16 crossings. 

Related SageMath programs with documentation are available on the author's GitHub site.
\end{abstract}

\maketitle

\section{Introduction}

Let $\widetilde{X}_K$ denote the infinite cyclic cover of the complement of a knot $K \subset S^3 $. The homology group $\hk$ is a finitely generated torsion module over the principal ideal domain $\qt$. Multiplication by $t$ corresponds to the map on homology induced by the deck transformation. For details on this and related aspects of knot theory, see a knot theory reference such as~\cite{MR3156509, MR521730, MR1472978}.

Every finitely generated modules over a \PID{} admits an \textit{invariant factor decomposition}. In particular, $\hk$ is the finite direct sum of cyclic modules:
\[ \hk \cong \frac{\qt}{\langle \delta_{K,1}(t) \rangle } \oplus \cdots \oplus \frac{\qt}{\langle \delta_{K,r}(t) \rangle } . \]
The fact that $\hk$ is torsion implies that $\langle \delta_i(t) \rangle $ is a proper ideal for all $i$ unless $\hk$ is trivial, in which case $r=1$ and $\delta_{K,1}(t) = 1$.
Henceforth, we simplify the notation by suppressing $K$ in subscripts. This decomposition is unique given the condition that for all positive $i$, with $i <r$, $\langle \delta_i(t) \rangle \subset \langle \delta_{i+1} (t) \rangle$.

For a pair of non-zero elements $f(t), g(t) \in \qt$, one has $\langle f (t) \rangle = \langle g(t) \rangle$ if and only if $f(t) = at^mg(t)$ for some $a \ne 0 \in \Q$ and $m \in \Z$. This permits us to choose a normalization for the generators. Henceforth, $\delta_i(t)$ will denote the generator satisfying the following properties: (1) $\delta_i(t) \in \Z[t]$ is primitive; and (2) $\delta_i(0) >0.$ The nesting of the ideals is equivalent to the condition that $\delta_{i+1}(t)$ divides $\delta_{i}(t)$ for $1 \le i <r$. In some references, the definition of the $\delta_i$ is left slightly ambiguous, defined to be in $\Z[t,t^{-1}]$ and well-defined up to multiplication by a unit, $\pm t^m$ for some $m \in \Z$.

\begin{definition} We define:
\begin{itemize}
\item The polynomials $\delta_i(t)$ are the \textit{invariant factors} of $K$.\smallskip
\item The \textit{Alexander polynomial} $\Delta(t)$ is $\Delta(t) = \prod_{j = 1}^r \delta_j(t)$.\smallskip
\item The \textit{higher order Alexander polynomials} are given by $\Delta_i(t) = \prod_{j = i}^r \delta_j(t)$.
\end{itemize}
\end{definition}

\begin{example} For the knot $8_{18}$ one has
\[ \hk \cong \frac{\qt}{\langle (t^2-t+1)(t^2-3t+1) \rangle } \oplus \frac{\qt}{\langle t^2 - t+1 \rangle},\]
so
\[ \Delta(t) = \Delta_1(t) = (t^2 - t+1)^2(t^2 - 3t +1)
\qquad \text{and} \qquad
\Delta_2(t) = t^2 - t +1 .\]

\end{example}

\noindent{\textbf{Convention.}} We define $\delta_i(t) = \Delta_i(t) = 1$ for all $i > r$. \smallskip

In general, the computation of the invariant factor decomposition, and thus the values of the $\Delta_i(t)$, is time-intensive: it entails diagonalizing a presentation matrix over the ring $\qt$. The exception is $\Delta_1(t) $, which can be computed as the determinant of a standard square presentation matrix $A_K$ of $\hk$. Beyond that, if standard diagonalization tools, such as \verb|smith_form()| in SageMath~\cite{SageMathManual}, are used, they begin to fail for knots with as few as 16 crossings. Our observation in this note is that much faster algebraic computations determine the values of the $\Delta_i(t)$ in the vast majority of cases. The following result, an immediate consequence of Theorem~\ref{thm-main} and Corollary~\ref{cor-main}, is sufficient to determine the Alexander polynomials of all 1.7 million prime knots of 16 or fewer crossings. In extensive testing of randomly generated knots of up to 100 crossings it has not failed to determine the higher order Alexander polynomials.

\begin{theorem}\label{thm0} Let $A_K$ be a square presentation matrix for $\hk$. Suppose that each irreducible factor $\phi_i$ of $\Delta(t) = \det(A_K)$ has exponent at most 6. Then the Alexander polynomials $\Delta_i(t)$ are determined by: (1) the ranks of the matrices formed by reducing $A_K$ to the number fields $\Q(\omega_i)$ where $\omega_i$ is a root of $\phi_i$; and (2) the $\LCM{}$ of the denominators of the reduced entries of the inverse $A_K^{-1} \in \GL(\Q(t))$.
\end{theorem}

\smallskip
\noindent{\textbf{Speed.}} A program that takes advantage of Theorem~\ref{thm0} computes the Alexander invariants of all roughly 3,000 prime knots of at most 12 crossings in under 3 seconds. The full list of prime knots of at most 13--crossing prime knots is computed in under 15 seconds. At 16 crossings, the program computes the Alexander invariants of the 1.7 million knots in just over an hour. This corresponds to a rate exceeding 400 knots per second. Random samples of high crossing knots were generated with SnapPy~\cite{snappy}; at 65 crossings the process slows down to about 1 second for each knot and at 100 crossings it takes a bit over 10 seconds per knot. Further details appear in Section~\ref{sec-compute} and Table~\ref{table-time}.

\smallskip
\noindent\textbf{The challenge of computing $\Delta_i(t)$.}
Computer algebra systems have built-in programs for diagonalizing a matrix $M$ over a polynomial ring. In SageMath, the
command M.\verb|smith_form()| yields the diagonalization $D$ along with matrices $U$ and $V$ for which $UMV = D$. The difficulty is that diagonalization is memory intensive and slow. For instance, for a specific 12--crossing knot with degree 6 Alexander polynomial, the diagonalizing matrices $U$ and $V$ produced by \verb|smith_form()| had entries that involve polynomials of degree up to 30 and the coefficients of these polynomials have numerators and denominators as large as $10^{70}$.

If one has a fast program for applying Seifert's algorithm to build a Seifert matrix $V_K$ from the PD notation for the knot, smaller matrices arise than if generated using Alexander's original approach. This results in some speed gain in the diagonalization process, but the Seifert matrices are no longer sparse. Initial experiments indicate that this approach might enable one to effectively compute the Alexander invariants of knots through perhaps 25 or 30 crossings using Sage.

\smallskip

\noindent{\em Available software.} Programs for computing Alexander invariants are available at~\cite{LivingstonGitHub}.

\smallskip

\noindent{\em Acknowledgements.}
The author made limited use of AI tools (ChatGPT, Perplexity) for proofreading and code-review assistance.


\section{Primary decompositions and the main theorem}

\subsection{Basic facts about the $\delta_i(t)$ and $\Delta_i(t)$}
The module $\hk$ has two standard presentation matrices. The first was given by Alexander~\cite{MR1501429}:
for an $(n+1)$--crossing diagram of $K$, the Alexander matrix $A_K$ is of size $n\times n$.  (In~\cite{MR1501429} the matrix is initially defined to be $(n+1) 
\times n$, and it is shown that any $n \times n$ submatrix is sufficient for computing the Alexander polynomial.  Seifert~\cite{MR1512955, MR0035436} developed an alternative presentation matrix: if $V$ is a Seifert matrix for $K$, then $V - tV^{\sf T}$ is a presentation matrix for $A_K$. In general, we will refer to any square presentation matrix of $ \hk$ as $A_K$. The size of $A_K$ will always be denoted $n \times n$. We will call any square presentation matrix an \textit{Alexander matrix}.

The next result states some elementary facts about the polynomials $\delta_i(t)$ and $\Delta_i(t)$.
\begin{theorem} For all $i$, the polynomials $\delta_i(t) \in \Z[t]$ and $\Delta_i(t) \in \Z[t]$ satisfy $ \delta_i(1) = \pm 1$ and $\Delta_i(1) = \pm 1$; in particular, they are primitive. They are symmetric:
\[\delta_i(t) = t^m\delta_i(t^{-1})\]
for some $m$ and $\Delta_i(t) = t^k\Delta_i(t^{-1})$ for some $k$.
\end{theorem}

\begin{proof} Work with the presentation matrix $A_K = V - tV^{\sf T}$ for a Seifert matrix $V$. The determinant $\det(V - V^{\sf T})$ is 1, and hence $\det(A_K) \in \Z[t]$ evaluates to 1 at $t=1$. In particular, it is primitive. The product of the $\delta_i(t)$ equals $\det(A_K)$ modulo a unit, of the form $\frac{a}{b}t^m$, so we find $\prod \delta_i(t) = \frac{a}{b}t^m \det(A_K)$ for some $a, b, m \in \Z$. Clearly, $b = \pm 1$. Since each $\delta_i(t)$ is by definition primitive, it must be that $a= \pm 1$. Thus, $\delta_i(1) = \pm 1$ for all $i$. Products of primitive polynomials are primitive, so each $\Delta_i(t)$ is primitive and $\Delta_i(1) = \pm 1$ for all $i$.

To prove symmetry, we proceed as follows. The primary decomposition theorem for a \PID{} can be proved by showing that there are invertible matrices $U, V\in \GL(\qt)$ such that $U A_K V$ is a diagonal matrix, $D$, having diagonal entries a set of polynomials $\delta_i(t)$ as above. It follows that $\Delta_1(t) = \pm t^k \det(A_K)$ for some integer $k$. It is evident that $\Delta_i(t)$ is the $\LCM$ of the $(n - i+1)$ minors of $D$. An easy linear algebra exercise proves that multiplication by invertible matrices, here $U$ and $V$, does not change the set of values of the minors.

We have
\[
(V - t V^{\sf T})^{\sf T} \;=\; t^n \bigl(V - t^{-1} V^{\sf T}\bigr).
\]
Transposing a matrix does not change its determinant or the collection of its minors of any given size. The symmetry of the $\Delta_i(t)$ is a consequence. That is,
\[\Delta_i(t) = t^m \Delta_i(t^{-1})
\] for some $m$. The symmetry of the $\delta_i(t)$ follows from this.
\end{proof}

\subsection{The primary decomposition} For any irreducible $\phi \in \qt$, the $\phi$--torsion in $\hk$ will be denoted $\hk_\phi$. If $\hk_\phi $ is nontrivial, then there is a unique primary decomposition
\[ \hk_\phi \cong \frac{\qt}{\langle \phi^{d_1} \rangle } \oplus \cdots \oplus \frac{\qt}{\langle \phi^{d_r} \rangle } \]
for some $r>0$ and
for some sequence $d_1 \ge d_2 \cdots \ge d_r >0$. (Note that $r$ depends on $\phi$.) Determining the $\Delta_i(t)$ is equivalent to determining the $d_i$ for all $\phi$. The next theorem is proved in Section~\ref{sec-proof}.
\begin{theorem} \label{thm-main} Let $K$ be a knot with Alexander polynomial $\Delta$, Alexander matrix $A_K$, and $\phi$--primary decomposition with exponents $(d_1, \ldots , d_r)$.
\begin{enumerate}

\item There is a quotient map $\qt \to \qt/\langle \phi \rangle = \Q(\omega)$, where $\omega$ is a root of $\phi$. Let $A_{K,\phi}$ denote the image of $A_K$ in $\GL(\Q(\omega))$. Let $r^*$ denote the dimension of the kernel of $A_{K,\phi}$, viewed as an automorphism of $\Q(\omega)^n$. Then in the primary decomposition, $r = r^*$.\smallskip\item Let $e^*$ denote the exponent of $\phi$ in $\Delta$. Then $\sum_{i=1}^r d_i = e^*$.\smallskip
\item Let $d_1^*$ denote the exponent of $\phi$ in the $\LCM$ of the denominators of the reduced entries of $A_K^{-1} \in \GL(\Q(t))$. Then in the primary decomposition, $d_1 = d_1^*$.
\end{enumerate}
In summary, the vector $(d_1 , \ldots , d_r)$ is a partition of $e^*$ of length $r^*$ with maximum term $d_1^*$.
\end{theorem}

\begin{corollary}\label{cor-main} If every irreducible factor of $\Delta$ has exponent at most 6, then Theorem~\ref{thm-main} determines all $\Delta_i(t)$.
\end{corollary}
\begin{proof} Suppose the exponent of $\phi$ is 6. There are 11 integer partitions of $6$:
\[ [6], [5,1], [4,2], [3,3], [4,1,1], [3,2,1], [2,2,2], [3,1,1,1], [2,2,1,1], [2,1,1,1,1], [1,1,1,1,1,1] .\]

For a given length and maximum value there is a unique element in this list.

For exponents less than 6 the argument is the same and the lists of possibilities are shorter.
\end{proof}

\subsection{Example} There are 2,977 prime knots of at most 12 crossings. The Alexander polynomials for each was determined by the conditions (1) and (2) of Theorem~\ref{thm-main} with four exceptions: $10_{a99}$, $ 12_{n508}$, $12_{n604}$, $12_{n666}$. For each of these, the Alexander polynomial is $(t^2 - t +1)^4$ and the associated kernels are two-dimensional. Thus we have either $\delta_1(t) = (t^2 - t +1)^3$ and $\delta_2(t) = (t^2 - t +1)$, or $\delta_1(t) = (t^2 - t +1)^2$ and $\delta_2(t) = (t^2 - t +1)^2$. In terms of the Alexander polynomials, this leaves undetermined whether $\Delta_2(t) = t^2 - t +1$ or $\Delta_2(t) = (t^2 - t +1)^2$. Condition (3) determines these: for $ 10_{a99}$, $\Delta_2(t) = (t^2-t+1)^2$ and for the other three, $\Delta_2(t) = t^2-t+1$.

\subsection{ Limitations of Theorem~\ref{thm-main}} There are two length 3 partitions of 7 with maximum value 3: $[3,3,1]$ and $[3,2,2]$. This presents the first setting in which Theorem~\ref{thm-main} does not determine the $\phi$ primary summand of $ \hk_\phi $. This does not arise for any prime knot with up to 16 crossings. More specifically, there are 1,701,935 prime knots with up to 16 crossings. Of these, all have $\Delta_4(t) = 1$ and 656 have $\Delta_3(t)$ nontrivial. Of these, every one had $\Delta_2(t) = (t^2-t+1)^2$ and $\Delta_3(t) = t^2-t+1$. Just over 1\% of the knots, 17,769 of them, have $\Delta_2(t) \ne 1$.

\subsection{Speed of computation}\label{sec-compute} Computing the $\Delta_i(t)$ by using a standard package to diagonalize $A_K$ begins to slow down at 12 crossings and begins to fail at 15 or 16 crossings. Programs built using the three conditions in Theorem~\ref{thm-main} computed the Alexander polynomials for the roughly 3,000 prime knots of at most 12 crossings in under 4 seconds, running SageMath on an M1 iMac. There are roughly 13,000 prime knots of at most 13 crossings; these took about 20 seconds. For the roughly 47,000 prime knot with crossing number exactly 14, the computation took about 95 seconds, a speed of roughly .002 seconds per knot.

To compare the speed of the algorithm for various crossing numbers, we used the complete list of 15 crossing prime knots and used SnapPy to generate random sets of 1,000 knots for various higher crossing number. We have the following.

\begin{table}[h!]\centering
\begin{tabular}{c|c}
\textbf{Crossings} & \textbf{Avg.~Time per Knot (secs.)} \\
\hline
15 & 0.0022\\
20 & 0.0055 \\
25 & 0.0132 \\
30 & 0.0277 \\
35 & 0.0520 \\
40 & 0.0925 \\
45 & 0.1503 \\
50 & 0.2367 \\
55 & 0.3574 \\
60 & 0.5442 \\
65 & 0.7835 \\
\end{tabular}
\caption{Average computation time for random knots of various crossing number.}
\label{table-time}

\end{table}

\section{Proof of Theorem~\ref{thm-main}}\label{sec-proof} There are three statements in the theorem. We address them with three lemmas.
\begin{lemma}There is a quotient map $\qt \to \qt/\langle \phi \rangle = \Q(\omega)$, where $\omega$ is root of $\phi$. Let $A_{K,\phi}$ denote the image of $A_K$ in $\GL(\Q(\omega))$. Let $r^*$ be the dimension of the kernel of $A_{K,\phi}$ viewed as an automorphism of $\Q(\omega)^n$. Then in the primary decomposition, $r = r^*$.
\end{lemma}

\begin{proof} There are invertible matrices $ U, V \in \GL(\qt)$ such that $UA_K V$ is in diagonal form. Conjugation by invertible matrices does not change the dimension of the kernel of the associated automorphism. It is evident from the diagonal form that this dimension is the number of diagonal entries divisible by $\phi$.
\end{proof}

\begin{lemma} Let $e^*$ denote the exponent of $\phi$ in $\Delta$. Then $\sum_{i=1}^r d_i = e^*$.
\end{lemma}

\begin{proof} This is immediate from the facts that $\prod_{i =1}^n \delta_i (t)= \Delta(t)$ and that $d_i$ is the exponent of $\phi$ in $\delta_i(t)$ for $1 \le i \le r$.
\end{proof}

\begin{lemma}
Let $d_1^*$ denote the exponent of $\phi$ in the $\LCM$ of the denominators of the reduced entries of $A_K^{-1} \in \GL(\Q(t))$. Then $d_1 = d_1^*$.
\end{lemma}

\begin{proof} The result is clear if $A_K$ equals its diagonalization $D$. In fact, the $\LCM$ of the denominators of $D^{-1}$ is precisely $\delta_1(t)$. The result now follows from the observation that if matrices $X, Y \in \text{GL}(\qt)$ are related by the equation $Y = UXV$ with $U, V \in \text{GL}(\qt)$ and both $U$ and $V$ have determinant 1, then the $\LCM$ of the denominators of $X^{-1}$ and $Y^{-1}$ are the same.

We have $Y^{-1} = V^{-1}X^{-1}U^{-1}$. Thus each entry of $Y^{-1}$ is a linear combination of elements in $X^{-1}$ with coefficients in $\qt$. The denominator of that entry in $Y^{-1}$ must divide the denominators of elements in $X^{-1}$. Thus, the $\LCM$ of the denominators of $Y^{-1}$ must divide the $\LCM$ of the denominators of $X^{-1}$. Reversing the roles of $X$ and $Y$ shows that the two $\LCM$s are the same.
\end{proof}

\section{Low crossing prime knots with nontrivial $\Delta_2(t)$ and $\Delta_3(t)$}
In~\cite[Table 1]{MR3156509}, all prime knots with up to 10 crossings and nontrivial $\Delta_2(t)$ are indicated. Table~\ref{11-table} presents all prime knots having up to 11 crossings for which $\Delta_2(t)$ is nontrivial, giving the values of $\Delta_1(t)$ and $\Delta_2(t)$.

No knots with $\Delta_3(t) \ne 1$ arise among prime knots in this range. In fact, the first knot for which $\Delta_3(t)$ is nontrivial is $14a_{1975}$, having $\Delta_1(t) = (5t^2 - 9t+5)(t^2-t+1)^3$, $\Delta_2(t) = (t^2-t+1)^2$, and $\Delta_3(t) = t^2 - t+1$. There are 1,701,935 prime knots with crossing number at most 16.

\begin{table}[h]
\centering
\renewcommand{\arraystretch}{1.3}
\setlength{\tabcolsep}{6pt}
\setlength{\arrayrulewidth}{0.3pt} 
\begin{tabular}{|l|l|l|}
\hline
{Knot} & ${\Delta_1(t)}$ & $ {\Delta_2(t)}$ \\
\hline
$8a_{18}$ & $(t^2 - 3t + 1)(t^2 - t + 1)^2$ & $t^2 - t + 1$ \\
\hline
$9a_{40}$ & $(t^2 - t + 1)(t^2 - 3t + 1)^2$ & $t^2 - 3t + 1$ \\
\hline
$10a_{98}$ & $(t - 2)(2t - 1)(t^2 - t + 1)^2$ & $t^2 - t + 1$ \\
\hline
$10a_{99}$ & $(t^2 - t + 1)^4$ & $(t^2 - t + 1)^2$ \\
\hline
$10a_{123}$& $(t^4 - 3t^3 + 3t^2 - 3t + 1)^2$ & $t^4 - 3t^3 + 3t^2 - 3t + 1$ \\
\hline
$11a_{43}$ & $(4t^2 - 7t + 4)(t^2 - t + 1)^2$ & $t^2 - t + 1$ \\
\hline
$11a_{44}$ & $(t^2 - t + 1)^2(t^4 - 3t^3 + 5t^2 - 3t + 1)$ & $t^2 - t + 1$ \\
\hline
$11a_{47}$ & $(t^2 - t + 1)^2(t^4 - 3t^3 + 5t^2 - 3t + 1)$ & $t^2 - t + 1$ \\
\hline
$11a_{57}$ & $(t^2 - t + 1)^2(t^4 - 3t^3 + 3t^2 - 3t + 1)$ & $t^2 - t + 1$ \\
\hline
$11a_{231}$ & $(t^2 - t + 1)^2(t^4 - 3t^3 + 3t^2 - 3t + 1)$ & $t^2 - t + 1$ \\
\hline
$11a_{263}$ & $(t^2 - t + 1)^2(2t^4 - 2t^3 + t^2 - 2t + 2)$ & $t^2 - t + 1$ \\
\hline
$11a_{297}$ & $(2t^2 - 3t + 2)(t^2 - 3t + 1)^2$ & $t^2 - 3t + 1$ \\
\hline
$11a_{332}$ & $(t^2 - t + 1)^2(t^4 - 5t^3 + 9t^2 - 5t + 1)$ & $t^2 - t + 1$ \\
\hline
$11n_{71}$ & $(2t^2 - 3t + 2)(t^2 - t + 1)^2$ & $t^2 - t + 1$ \\
\hline
$11n_{72}$ & $(t - 2)(2t - 1)(t^2 - t + 1)^2$ & $t^2 - t + 1$ \\
\hline
$11n_{73}$ & $(t^2 - t + 1)^2$ & $t^2 - t + 1$ \\
\hline
$11n_{74}$ & $(t^2 - t + 1)^2$ & $t^2 - t + 1$ \\
\hline
$11n_{75}$ & $(2t^2 - 3t + 2)(t^2 - t + 1)^2$ & $t^2 - t + 1$ \\
\hline
$11n_{76}$ & $(t^2 - t + 1)^2(t^4 - t^3 + t^2 - t + 1)$ & $t^2 - t + 1$ \\
\hline
$11n_{77}$ & $(t^2 - t + 1)^2(t^4 + t^3 - 3t^2 + t + 1)$ & $t^2 - t + 1$ \\
\hline
$11n_{78}$ & $(t^2 - t + 1)^2(t^4 - t^3 + t^2 - t + 1)$ & $t^2 - t + 1$ \\
\hline
$11n_{81}$ & $(t^2 - t + 1)^2(t^4 - t^3 - t^2 - t + 1)$ & $t^2 - t + 1$ \\
\hline
$11n_{164}$ & $(t^2 - 3t + 1)(t^2 - t + 1)^2$ & $t^2 - t + 1$ \\
\hline
\end{tabular}
\caption{Prime knots with nontrivial $\Delta_2(t)$ (with corresponding $\Delta_1(t)$ values).}
\label{11-table}
\end{table}


\bibliographystyle{amsplain}
\bibliography{../BibTexComplete.bib}

\end{document}